\documentclass{amsart}
\usepackage{mathrsfs}

\newtheorem{theorem}{Theorem}[section]
\newtheorem{lemma}[theorem]{Lemma}
\newtheorem{proposition}[theorem]{Proposition}
\newtheorem{corollary}[theorem]{Corollary}

\theoremstyle{definition}

\theoremstyle{remark}
\newtheorem{remark}[theorem]{Remark}

\numberwithin{equation}{section}

\begin{document}

\title{Inequalities for modified Bessel functions and their integrals}

\author{Robert E. Gaunt}
\address{Department of Statistics, University of Oxford, Oxford, UK}
\email{gaunt@stats.ox.ac.uk}
\thanks{The author is supported by EPSRC research grant AMRYO100.}


\date{June 2014.}


\begin{abstract}Simple inequalities for some integrals involving the modified Bessel functions $I_{\nu}(x)$ and $K_{\nu}(x)$ are established.  We also obtain a monotonicity result for $K_{\nu}(x)$ and a new lower bound, that involves gamma functions, for $K_0(x)$. 
\end{abstract}

\maketitle

\section{Introduction and preliminary results}
In the developing Stein's method for Variance-Gamma distributions, Gaunt \cite{me} required simple bounds, in terms of modified Bessel functions, for the integrals
\[\int_0^x\mathrm{e}^{\beta t}t^{\nu}I_{\nu}(t)\,\mathrm{d}t \qquad\mbox{and} \qquad \int_x^{\infty}\mathrm{e}^{\beta t}t^{\nu}K_{\nu}(t)\,\mathrm{d}t,\]
where $x>0$, $\nu>-1/2$ and $-1<\beta<1$.  Closed form expressions for these integrals, in terms of modified Bessel functions and the modified Struve function $\mathbf{L}_{\nu}(x)$, do in fact exist for the case $\beta=0$.  For $z\in\mathbb{C}$ and $\nu\in\mathbb{C}$, let $\mathscr{L}_{\nu}(z)$ denote $I_{\nu}(z)$, $ \mathrm{e}^{\nu \pi i}K_{\nu}(z)$ or any linear combination of these functions, in which the coefficients are independent of $\nu$ and $z$.  From formula 10.43.2 of Olver et$.$ al$.$ \cite{olver} we have, for $\nu \not= -1/2$,
\begin{equation}\label{besint}\int z^{\nu}\mathscr{L}_{\nu}(z)\,\mathrm{d}z =\sqrt{\pi}2^{\nu-1}\Gamma(\nu+1/2)z(\mathscr{L}_{\nu}(z)\mathbf{L}_{\nu-1}(z)-\mathscr{L}_{\nu-1}(z)\mathbf{L}_{\nu}(z)). 
\end{equation}
Whilst formula (\ref{besint}) holds for complex-valued $z$ and $\nu$, throughout this paper we shall restrict our attention to the case of real-valued $z$ and $\nu$.  There are no closed form expressions in terms of modified Bessel and Struve functions in the literature for the integrals for the case $\beta\not=0$.  Moreover, even in the case $\beta=0$ the expression on the right-hand side of formula (\ref{besint}) is a complicated expression involving the modified Struve function $\mathbf{L}_{\nu}(x)$.  This provides the motivation for establishing simple bounds, in terms of modified Bessel functions, for the integrals defined in the first display. 

In this paper we establish, through the use of elementary properties of modified Bessel functions and straightforward calculations, simple bounds, that involve modified Bessel functions, for the integrals given in the first display.  Our bounds prove to be very useful when applied to calculations that arise in the study of Stein's method for Variance-Gamma distributions.  We also obtain a monotonicity result and bound for the modified Bessel function of the second kind $K_{\nu}(x)$, as well as a simple but remarkably tight lower bound for $K_0(x)$.  These bounds are, again, motivated by the need for such bounds in the study of Stein's method for Variance-Gamma distributions.  However, the bounds obtained in this paper may also prove to be useful in other problems related to modified Bessel functions; see for example, Baricz and Sun \cite{baricz3} in which inequalities for modified Bessel functions of the first kind were used to obtain lower and upper bounds for integrals of involving modified Bessel functions of the first kind.    Throughout this paper we make use of some elementary properties of modified Bessel functions and these are stated in the appendix.

\section{Inequalities for integrals involving modified Bessel functions}

Before presenting our first result concerning inequalities for integrals of modified Bessel functions, we introduce some notation for the repeated integral of the function $\mathrm{e}^{\beta x}x^{\nu}I_{\nu}(x)$, which will be used in the following theorem.  We define
\begin{equation} \label{super1} I_{(\nu,\beta,0)}(x)=\mathrm{e}^{\beta x}x^{\nu}I_{\nu}(x), \quad I_{(\nu,\beta,n+1)}(x)=\int_0^xI_{(\nu,\beta,n)}(y)\,\mathrm{d}y,\quad n =0,1,2,3,\ldots.
\end{equation}
With this notation we have:

\begin{theorem} \label{tiger} Let $0\leq \gamma <1$, then the following inequalities hold for all $x>0$
\begin{equation}\label{lowerint}\int_0^xt^{\nu}I_{\nu}(t)\,\mathrm{d}t > x^{\nu}I_{\nu+1}(x), \quad \nu >-1,
\end{equation}
\begin{equation}\label{100fcp}\int_0^xt^{\nu}I_{\nu}(t)\,\mathrm{d}t <x^{\nu}I_{\nu}(x), \quad \nu\geq 1/2,
\end{equation}
\begin{equation}\label{2003001}I_{(\nu,0,n+1)}(x) < I_{(\nu,0,n)}(x), \quad \nu \geq 1/2,
\end{equation}
\begin{equation}\label{45210}I_{(\nu,-\gamma,n)}(x) \leq \frac{1}{(1-\gamma)^n}\mathrm{e}^{-\gamma x}I_{(\nu,0,n)}(x), \quad \nu \geq 1/2, \: n=0,1,2,\ldots,
\end{equation}
\begin{equation}\label{intIi} \int_0^x t^{\nu} I_{\nu +n}(t)\,\mathrm{d}t < \frac{2(\nu + n +1)}{2\nu +n +1} x^{\nu} I_{\nu + n+1}(x), \quad \nu >-1/2, \: n\ge 0,
\end{equation}
\begin{equation}\label{intIii}I_{(\nu,0,n)}(x) < \bigg\{\prod_{k=1}^n \frac{2\nu + 2k}{2\nu + k}\bigg\} x^{\nu} I_{\nu + n}(x), \quad \nu \geq 0, \: n=1,2,3\ldots,
\end{equation} 
\begin{equation*} \label{owl} I_{(\nu,-\gamma,n)}(x) < \frac{1}{(1-\gamma)^n}\bigg\{\prod_{k=1}^n \frac{2\nu + 2k}{2\nu + k}\bigg\} \mathrm{e}^{-\gamma x}x^{\nu} I_{\nu + n}(x), \quad \nu \geq 1/2, \: n=1,2,3,\ldots.
\end{equation*}
\end{theorem}

\begin{proof}
(i) From the differentiation formula (\ref{diffone}) we have that
\[\int_0^xt^{\nu}I_{\nu}(t)\,\mathrm{d}t =\int_0^x\frac{1}{t}t^{\nu+1}I_{\nu}(t)\,\mathrm{d}t >\frac{1}{x}\int_0^xt^{\nu+1}I_{\nu}(t)\,\mathrm{d}t =x^{\nu}I_{\nu+1}(x),\]
since by (\ref{Itend0}) we have $\lim_{x\downarrow 0}x^{\nu+1}I_{\nu+1}(x)=0$ for $\nu>-1$.

(ii) Using inequality (\ref{Imon}) and then applying (\ref{diffone}) we get
\[\int_0^xt^{\nu}I_{\nu}(t)\,\mathrm{d}t <\int_0^xt^{\nu}I_{\nu-1}(t)\,\mathrm{d}t=x^{\nu}I_{\nu}(x).\]

(iii) From inequality (\ref{100fcp}), we have 
\[I_{(\nu,0,1)}(x)<I_{(\nu,0,0)}(x).\]
Integrating both sides of the above display $n$ times with respect to $x$ yields the desired inequality.

(iv) We prove the result by induction on $n$.  The result is trivially true for $n=0$.  Suppose the result is true for $n=k$.  From the inductive hypothesis we have
\begin{equation} \label{thing} I_{(\nu,-\gamma,k+1)}(x) = \int_0^x I_{(\nu,-\gamma,k)}(t)\,\mathrm{d}t \leq \frac{1}{(1-\gamma)^k}\int_0^x \mathrm{e}^{-\gamma t}I_{(\nu,0,k)}(t)\,\mathrm{d}t.
\end{equation}
Integration by parts and inequality (\ref{2003001}) gives
\begin{align*} \int_0^x \mathrm{e}^{-\gamma t} I_{(\nu,0,k)}(t) \,\mathrm{d}t &= \mathrm{e}^{-\gamma x}I_{(\nu,0,k+1)}(x) + \gamma \int_0^x \mathrm{e}^{-\gamma t} I_{(\nu,0,k+1)}(t) \,\mathrm{d}t \\ &< \mathrm{e}^{-\gamma x}I_{(\nu,0,k+1)}(x) + \gamma \int_0^x \mathrm{e}^{-\gamma t} I_{(\nu,0,k)}(t) \,\mathrm{d}t. 
\end{align*}
Rearranging we obtain
\[\int_0^x \mathrm{e}^{-\gamma t} I_{(\nu,0,k)}(t) \,\mathrm{d}t < \frac{1}{1-\gamma} \mathrm{e}^{-\gamma x}I_{(\nu,0,k+1)}(x),\]
and substituting into (\ref{thing}) gives 
\[I_{(\nu,-\gamma,k+1)}(x) < \frac{1}{(1-\gamma)^{k+1}}\mathrm{e}^{-\gamma x}I_{(\nu,0,k+1)}(x).\]
Hence the result has been proved by induction.

(v) From the differentiation formula (\ref{diffone}) and identity (\ref{Iidentity}) we get that
\begin{align*} \frac{\mathrm{d}}{\mathrm{d}t} (t^{\nu} I_{\nu +n+1} (t)) &= \frac{\mathrm{d}}{\mathrm{d}t}(t^{-(n+1)} \cdot t^{\nu +n+1} I_{\nu+n+1} (t)) \\
& = t^{\nu} I_{\nu +n} (t) -(n+1)t^{\nu -1} I_{\nu +n+1}(t) \\
& = t^{\nu} I_{\nu +n} (t) - \frac{n+1}{2(\nu +n+1)} t^{\nu} I_{\nu+n} (t) + \frac{n+1}{2(\nu +n+1)} t^{\nu} I_{\nu +n+2} (t) \\
& = \frac{2\nu +n+1}{2(\nu +n+1)} t^{\nu} I_{\nu +n} (t) + \frac{n+1}{2(\nu +n+1)} t^{\nu} I_{\nu +n+2} (t). 
\end{align*}
Integrating both sides over $(0,x)$, applying the fundamental theorem of calculus and rearranging gives
\[\int_0^x t^{\nu} I_{\nu +n} (t)\,\mathrm{d}t = \frac{2(\nu +n+1)}{2\nu +n+1} x^{\nu} I_{\nu +n+1} (x) - \frac{n+1}{2\nu +n+1} \int_0^x t^{\nu} I_{\nu +n +2} (t)\,\mathrm{d}t.\]
The result now follows from the fact that $I_{\nu} (x) > 0$ for $x > 0$ and by the positivity of the integral.

(vi) From inequality (\ref{intIi}) we have
\[I_{(\nu,0,1)} (x) = \int_0^x t^{\nu} I_{\nu} (t)\,\mathrm{d}t < \frac{2(\nu +1)}{2\nu +1} x^{\nu} I_{\nu+1} (x),\]
and
\begin{align*}
I_{(\nu,0,2)} (x) &= \int_0^x I_{(\nu,0,1)} (t)\,\mathrm{d}t \\
&< \frac{2(\nu +1)}{2\nu +1} \int_0^x t^{\nu} I_{\nu+1} (t)\,\mathrm{d}t \\
&< \frac{2(\nu +1)}{2\nu +1} \frac{2(\nu +2)}{2\nu +2} x^{\nu} I_{\nu +2} (x).
\end{align*}
Iterating gives the result.

(vii) This follows from inequalities (\ref{45210}) and (\ref{intIii}).
\end{proof}

We now state a simple lemma (which is a special case of Lemma 2.4 of Ismail and Muldoon \cite{ismail1}), that gives a monotonicity result for the ratio $\frac{K_{\nu-1}(x)}{K_{\nu}(x)}$.  The lemma has an immediate corollary, which we will make use of in the proof of our next theorem.

\begin{lemma}\label{bobcatttt}Suppose $x>0$, then the function $\frac{K_{\nu-1}(x)}{K_{\nu}(x)}$ is strictly monotone increasing for $\nu>1/2$, is constant for $\nu=1/2$, and is strictly monotone decreasing for $\nu<1/2$.  
\end{lemma}

\begin{corollary}\label{bobcat}For $\nu>1/2$ and $\alpha>1$ the equation $K_{\nu}(x)=\alpha K_{\nu-1}(x)$ has one root in the region $x>0$.
\end{corollary}

\begin{proof}
From the asymptotic formulas (\ref{Ktend0}) and (\ref{Ktendinfinity}), it follows that for $\nu>1/2$,
\[\lim_{x\downarrow 0}\frac{K_{\nu-1}(x)}{K_{\nu}(x)}=0, \qquad \mbox{and} \qquad \lim_{x\rightarrow\infty}\frac{K_{\nu-1}(x)}{K_{\nu}(x)}=1. \]
Since $\frac{K_{\nu-1}(x)}{K_{\nu}(x)}$ is strictly monotone increasing on $(0,\infty)$, it follows that for $\alpha>1$ the equation $K_{\nu}(x)=\alpha K_{\nu-1}(x)$ (i.e. $\frac{K_{\nu-1}(x)}{K_{\nu}(x)}=\frac{1}{\alpha}$) has one root in the region $x>0$.  
\end{proof}

As an aside, we note that Lemma \ref{bobcat} allows us to easily establish an inequality for the Tur\'{a}nian $\Delta_{\nu}(x)=K_{\nu}^2(x)-K_{\nu-1}(x)K_{\nu+1}(x)$ (for more details on the Tur\'{a}nian $\Delta_{\nu}(x)$ see Baricz \cite{baricz1}).

\begin{proposition}Suppose $x>0$, then $\Delta_{\nu}(x)<\Delta_{\nu-1}(x)$ for $\nu>1/2$, $\Delta_{1/2}(x)=\Delta_{-1/2}(x)$, and $\Delta_{\nu}(x)>\Delta_{\nu-1}(x)$ for $\nu<1/2$.
\end{proposition}

\begin{proof}By the quotient rule and differentiation formula (\ref{cat}), we have
\begin{align*}\frac{\mathrm{d}}{\mathrm{d}x}\bigg(\frac{K_{\nu-1}(x)}{K_{\nu}(x)}\bigg)&=-\frac{(K_{\nu}(x)+K_{\nu-2}(x))K_{\nu}(x)-(K_{\nu+1}(x)+K_{\nu-1}(x))K_{\nu-1}(x)}{2K_{\nu}^2(x)} \\
&=\frac{K_{\nu-1}^2(x)-K_{\nu-2}(x)K_{\nu}(x)-(K_{\nu}^2(x)-K_{\nu-1}(x)K_{\nu+1}(x))}{2K_{\nu}^2(x)} \\
&=\frac{\Delta_{\nu-1}(x)-\Delta_{\nu}(x)}{2K_{\nu}^2(x)}.
\end{align*}
Since, by Lemma \ref{bobcat}, the function $\frac{K_{\nu-1}(x)}{K_{\nu}(x)}$ is strictly monotone increasing for $\nu>1/2$, is constant for $\nu=1/2$, and is strictly monotone decreasing for $\nu<1/2$, the result follows.
\end{proof}

With the aid of Corollary \ref{bobcat} and standard properties of the modified Bessel function $K_{\nu}(x)$, we can prove at the following theorem. 

\begin{theorem}\label{tiger1}Let $-1< \beta <1$, then for all $x>0$ the following inequalities hold
\begin{eqnarray}\int_x^{\infty} t^{\nu} K_{\nu}(t)\,\mathrm{d}t &<& x^{\nu} K_{\nu +1}(x), \quad \nu\in\mathbb{R}, \nonumber \\
 \label{pizzaz} \int_x^{\infty} t^{\nu} K_{\nu}(t)\,\mathrm{d}t &<& x^{\nu} K_{\nu }(x), \quad \nu < 1/2, \\
 \label{KBinteqi} \int_x^{\infty}\mathrm{e}^{\beta t}t^{\nu}K_{\nu}(t)\,\mathrm{d}t &<& \frac{1}{1-|\beta|}\mathrm{e}^{\beta x}x^{\nu}K_{\nu}(x), \quad \nu < 1/2, \\
 \label{gggh} \int_x^{\infty} t^{\nu} K_{\nu}(t)\,\mathrm{d}t &\leq& \frac{ \sqrt{\pi} \Gamma(\nu +1/2)}{\Gamma(\nu)} x^{\nu} K_{\nu}(x), \quad \nu \geq 1/2, \\
 \label{neo mice} \int_x^{\infty} \mathrm{e}^{\beta t} t^{\nu} K_{\nu}(t)\,\mathrm{d}t &\leq& \frac{ 2\sqrt{\pi} \Gamma(\nu +1/2)}{(1-\beta^2)^{\nu+1/2}\Gamma(\nu)} \mathrm{e}^{\beta x}x^{\nu} K_{\nu}(x), \quad \nu \geq 1/2 . \nonumber
\end{eqnarray}
\end{theorem}

\begin{proof}
(i) From the differentiation formula (\ref{diffKi}) we have that 
\[\int_x^{\infty} t^{\nu} K_{\nu}(t)\,\mathrm{d}t = \int_x^{\infty} \frac{1}{t} t^{\nu+1} K_{\nu}(t)\,\mathrm{d}t < \frac{1}{x} \int_x^{\infty} t^{\nu +1} K_{\nu}(t)\,\mathrm{d}t = x^{\nu} K_{\nu +1} (x),\]
since, by the asymptotic formula (\ref{Ktendinfinity}), $\lim_{x\rightarrow \infty}x^{\nu+1}K_{\nu+1}(x)=0$.

(ii) Using inequality (\ref{Kmoni}) and then apply the differentiation formula (\ref{diffKi}) we have
\[\int_x^{\infty} t^{\nu} K_{\nu}(t)\,\mathrm{d}t< \int_x^{\infty} t^{\nu} K_{\nu-1}(t)\,\mathrm{d}t=x^{\nu}K_{\nu}(x).\]

(iii) Now suppose that $\nu< 1/2$ and $\beta > 0$.  Using integration by parts and the differentiation formula (\ref{diffKi}) gives
\begin{equation}\label{annoy}\int_x^{\infty}\mathrm{e}^{\beta t}t^{\nu}K_{\nu}(t)\,\mathrm{d}t =-\frac{1}{\beta}\mathrm{e}^{\beta x}x^{\nu}K_{\nu}(x) +\frac{1}{\beta}\int_x^{\infty}\mathrm{e}^{\beta t} t^{\nu}K_{\nu-1}(t)\,\mathrm{d}t.
\end{equation}
Applying the inequality (\ref{Kmoni}) and rearranging gives
\[\left(\frac{1}{\beta}-1\right)\int_x^{\infty}\mathrm{e}^{\beta t}t^{\nu}K_{\nu}(t)\,\mathrm{d}t< \frac{1}{\beta}\mathrm{e}^{\beta x}x^{\nu}K_{\nu}(x).\]
Inequality (\ref{KBinteqi}) for $\beta>0$ now follows on rearranging. 

The case $\beta\leq 0$ is simple.  Since $\mathrm{e}^{\beta t}$ is a non increasing function of $t$ when $\beta \leq 0$ we have
\[\int_x^{\infty}\mathrm{e}^{\beta t}t^{\nu}K_{\nu}(t)\,\mathrm{d}t \leq \mathrm{e}^{\beta x}\int_x^{\infty}t^{\nu}K_{\nu}(t)\,\mathrm{d}t < \mathrm{e}^{\beta x}x^{\nu}K_{\nu}(x)\leq \frac{1}{1-|\beta |}\mathrm{e}^{\beta x}x^{\nu}K_{\nu}(x),\]
where we used inequality (\ref{pizzaz}) to obtain the second inequality.  Hence inequality (\ref{KBinteqi}) has been proved.

(iv) The case $\nu = 1/2$ is simple.  Using (\ref{sphere}) we may easily integrate $t^{1/2}K_{1/2}(t)$:
  \[\int_x^{\infty}t^{1/2}K_{1/2}(t)\,\mathrm{d}t =\int_x^{\infty}\sqrt{\frac{\pi}{2}}\mathrm{e}^{-t}\,\mathrm{d}t=\sqrt{\frac{\pi}{2}}\mathrm{e}^{-x}= x^{1/2}K_{1/2}(x).\]
It therefore follows that inequality (\ref{gggh}) holds for $\nu=1/2$ because we have
\[\frac{\sqrt{\pi}\Gamma(1)}{\Gamma(1/2)}=1,\]
where we used the facts that $\Gamma(1)=1$ and $\Gamma(1/2)=\sqrt{\pi}$.

Now suppose $\nu>1/2$.  We begin by defining the function $u(x)$ to be
\begin{equation*}u(x)= M x^{\nu} K_{\nu}(x) - \int_x^{\infty}  t^{\nu} K_{\nu}(t)\,\mathrm{d}t, 
\end{equation*}
where
\[M=\frac{ \sqrt{\pi} \Gamma(\nu +1/2)}{\Gamma(\nu)}.\]
We now show that $u(x)\geq 0$ for all $x \geq 0$, which will prove the result.  We begin by noting that $\lim\nolimits_{x\rightarrow 0^{+}}u(x)=0$ and $\lim\nolimits_{x\rightarrow \infty}u(x)=0$, which are verified by the following calculations, where we make use of the asymptotic formula (\ref{Ktend0}) and the definite integral formula (\ref{pdfk}).
\begin{align*}u(0)&=\lim_{x\to 0^{+}}\frac{ \sqrt{\pi} \Gamma(\nu +1/2)}{\Gamma(\nu)}  x^{\nu} K_{\nu}(x)-\int_0^{\infty}  t^{\nu} K_{\nu}(t)\,\mathrm{d}t \\
&=\sqrt{\pi} \Gamma(\nu +1/2)2^{\nu-1} -\sqrt{\pi} \Gamma(\nu +1/2)2^{\nu-1} \\
&=0,
\end{align*}
and
\[\lim_{x\rightarrow \infty}u(x)=\lim_{x\to \infty}M  x^{\nu} K_{\nu}(x)-\lim_{x\to \infty}\int_x^{\infty}  t^{\nu} K_{\nu}(t)\,\mathrm{d}t =0,\]
where we used the asymptotic formula (\ref{Ktendinfinity}) to obtain the above equality.  We may obtain an expression for the first derivative of $u(x)$ by the use of the differentiation formula (\ref{diffKi}) as follows
\begin{equation}\label{uder}u'(x)=x^{\nu} [K_{\nu}(x) -MK_{\nu-1}(x)].
\end{equation}
In the limit $x\to\ 0^{+}$ we have, by the asymptotic formula (\ref{Ktend0}), that
\begin{equation*}u'(x) \sim\begin{cases} x^{\nu}\left\{2^{\nu-1}\Gamma(\nu)\frac{1}{x^{\nu}} -M2^{|\nu-1|-1}\Gamma(|\nu-1|)\frac{1}{x^{|\nu-1|}}\right\}, 
& \text{$\nu \not= 1$}, \\
x^{\nu}\left\{2^{\nu-1}\Gamma(\nu)\frac{1}{x^{\nu}} +M\log x\right\}, 
& \text{$\nu = 1$}.
\end{cases}
\end{equation*}
Since $\nu>|\nu-1|$ for $\nu>1/2$ and $\lim\nolimits_{x\rightarrow 0^{+}}x^a\log x=0$, where $a>0$, we have
\[u'(x) \sim 2^{\nu-1}\Gamma(\nu),\quad \mbox{as }x\rightarrow 0^{+}, \quad \mbox{for } \nu >1/2.\] 
Therefore $u(x)$ is initially an increasing function of $x$.  In the limit $x\to \infty$ we have, by (\ref{Ktendinfinity}), 
\[u'(x) \sim \left(1-\frac{ \sqrt{\pi} \Gamma(\nu +1/2)}{\Gamma(\nu)} \right)\sqrt{\frac{\pi}{2}}x^{\nu-1/2} \mathrm{e}^{-x} <0, \quad \mbox{for }\nu >1/2.\]
We therefore see that $u(x)$ is an decreasing function of $x$ for large, positive $x$.  From the formula (\ref{uder}) we see that $x^*$ is a turning point of $u(x)$ if and only if
\begin{equation} \label{toad} K_{\nu}(x^*) =\frac{ \sqrt{\pi} \Gamma(\nu +1/2)}{\Gamma(\nu)}K_{\nu-1}(x^*).
\end{equation}
From Corollary \ref{bobcat}, it follows that equation (\ref{toad}) has one root for $\nu>1/2$ (for which $\frac{ \sqrt{\pi} \Gamma(\nu +1/2)}{\Gamma(\nu)}>1$).

Putting these results together, we see that $u(x)$ is non-negative at the origin and initially increases until it reaches it maximum value at $x^*$, it then decreases and tends to $0$ as $x \rightarrow \infty$.  Therefore $u(x)$ is non-negative for all $x\geq 0$ when $\nu >1/2$.

(v) The proof for $\beta \leq 0$ is easy and follows immediately from part (iv), since $1<\frac{2}{(1-\beta^2)^{\nu+1/2}}$ for $\nu\geq 1/2$.  So we suppose $\beta >0$.  Again, because $K_{1/2}(x)=\sqrt{\frac{\pi}{2x}}\mathrm{e}^{-x}$, the case $\nu=1/2$ is straightforward, so we also suppose $\nu>1/2$.  We make use of a similar argument to the one used in the proof of part (iv).   We define the function $v(x)$ to be
\[v(x)=N\mathrm{e}^{\beta x}x^{\nu}K_{\nu}(x)-\int_x^{\infty} \mathrm{e}^{\beta t} t^{\nu} K_{\nu}(t)\,\mathrm{d}t,\]
where
\[N=\frac{2\sqrt{\pi}\Gamma(\nu+1/2)}{(1-\beta^2)^{\nu+1/2}\Gamma(\nu)}.\]
We now show that $v(x)\geq 0$ for all $x \geq 0$, which will prove the result.  We begin by noting that $\lim\nolimits_{x\rightarrow 0^{+}}v(x)>0$ and $\lim\nolimits_{x\rightarrow \infty}v(x)=0$, which are verified by the following calculations, where we make use of the asymptotic formula (\ref{Ktend0}) and the definite integral formula (\ref{pdfk}).
\begin{align*}v(0)&=\lim_{x\to 0^{+}}\frac{ 2\sqrt{\pi} \Gamma(\nu +1/2)}{(1-\beta^2)^{\nu+1/2}\Gamma(\nu)}  x^{\nu} K_{\nu}(x)-\int_0^{\infty}  \mathrm{e}^{\beta t}t^{\nu} K_{\nu}(t)\,\mathrm{d}t \\
&=\frac{ 2\sqrt{\pi} \Gamma(\nu +1/2)}{(1-\beta^2)^{\nu+1/2}\Gamma(\nu)}\cdot 2^{\nu-1}\Gamma(\nu) -\int_{0}^{\infty}  \mathrm{e}^{\beta t}t^{\nu} K_{\nu}(t)\,\mathrm{d}t\\
&>\frac{ 2\sqrt{\pi} \Gamma(\nu +1/2)}{(1-\beta^2)^{\nu+1/2}\Gamma(\nu)}\cdot 2^{\nu-1}\Gamma(\nu) -\int_{-\infty}^{\infty}  \mathrm{e}^{\beta t}|t|^{\nu} K_{\nu}(|t|)\,\mathrm{d}t\\
&= \frac{\sqrt{\pi} \Gamma(\nu +1/2)2^{\nu}}{(1-\beta^2)^{\nu+1/2}}-\frac{\sqrt{\pi} \Gamma(\nu +1/2)2^{\nu}}{(1-\beta^2)^{\nu+1/2}} \\
&=0,
\end{align*}
and
\[\lim_{x\rightarrow \infty}v(x)=\lim_{x\to \infty}N\mathrm{e}^{\beta x}  x^{\nu} K_{\nu}(x)-\lim_{x\to \infty}\int_x^{\infty}  \mathrm{e}^{\beta t}t^{\nu} K_{\nu}(t)\,\mathrm{d}t =0,\]
where we used the asymptotic formula (\ref{Ktendinfinity}) to obtain the above equality.  We may obtain an expression for the first derivative of $v(x)$ by the use of the differentiation formula (\ref{diffKi}) as follows
\begin{equation}\label{uder1}v'(x)=\mathrm{e}^{\beta x}x^{\nu} [(1+N\beta)K_{\nu}(x) -NK_{\nu-1}(x)].
\end{equation}
In the limit $x\to\ 0^{+}$ we have, by the asymptotic formula (\ref{Ktend0}), that
\begin{equation*}v'(x) \sim\begin{cases} \mathrm{e}^{\beta x}x^{\nu}\left\{2^{\nu-1}\Gamma(\nu)(1+N\beta)\frac{1}{x^{\nu}} -N\cdot2^{|\nu-1|-1}\Gamma(|\nu-1|)\frac{1}{x^{|\nu-1|}}\right\}, 
& \text{$\nu \not= 1$}, \\
\mathrm{e}^{\beta x}x^{\nu}\left\{2^{\nu-1}\Gamma(\nu)(1+N\beta)\frac{1}{x^{\nu}} +N\log x\right\}, 
& \text{$\nu = 1$}.
\end{cases}
\end{equation*}
As in part (iv), we see that $v(x)$ is initially an increasing function of $x$.  In the limit $x\to \infty$ we have 
\[v'(x) \sim (1-N(1-\beta))\sqrt{\frac{\pi}{2}}x^{\nu-1/2} \mathrm{e}^{(\beta-1)x}, \quad \mbox{for }\nu >1/2.\]
Now, for $\nu>1/2$ and $0<\beta<1$ we have, by (\ref{Ktendinfinity}),
\begin{equation}\label{curgap}N(1-\beta)=\frac{2\sqrt{\pi}\Gamma(\nu+1/2)}{\Gamma(\nu)}\cdot\frac{1}{(1-\beta^2)^{\nu-1/2}}\cdot\frac{1}{1+\beta}>2\cdot1\cdot\frac{1}{2}=1.
\end{equation}
Hence, $v(x)$ is an decreasing function of $x$ for large, positive $x$.  From formula (\ref{uder1}) we see that $x^*$ is a turning point of $v(x)$ if and only if
\begin{equation} \label{toad1} (1+N\beta)K_{\nu}(x^*) =NK_{\nu-1}(x^*).
\end{equation}
Inequality (\ref{curgap}) shows that $N>1+N\beta$ for all $\nu>1/2$ and $0<\beta<1$.  From Corollary \ref{bobcat}, it follows that equation (\ref{toad1}) has one root for positive $x$ and therefore $v(x)$ has one maximum which occurs at positive $x$.  Putting these results together we see that $v(x)$ is positive at the origin and initially increases until it reaches it maximum value at $x^*$, it then decreases and tends to $0$ as $x \rightarrow \infty$.  Therefore $v(x)$ is non-negative for all $x\geq 0$ when $\nu >1/2$, which  completes the proof.
\end{proof}

Combining the inequalities of Theorems \ref{tiger} and \ref{tiger1} and the indefinite integral formula (\ref{besint}) we may obtain lower and upper bounds for the quantity $\mathscr{L}_{\nu}(x)\mathbf{L}_{\nu-1}(x)-\mathscr{L}_{\nu-1}(x)\mathbf{L}_{\nu}(x)$.  Here is an example:

\begin{corollary}\label{struvebessel}Suppose $\nu>-1/2$, then for all $x>0$ we have
\[\frac{x^{\nu-1}I_{\nu+1}(x)}{\sqrt{\pi}2^{\nu-1}\Gamma(\nu+1/2)}<I_{\nu}(x)\mathbf{L}_{\nu-1}(x)-I_{\nu-1}(x)\mathbf{L}_{\nu}(x)<\frac{(\nu+1)x^{\nu-1}I_{\nu+1}(x)}{\sqrt{\pi}2^{\nu-1}\Gamma(\nu+3/2)}.\]
\end{corollary}

\begin{proof}From the asymptotic formulas (\ref{Itend0}) and (\ref{struve0}) for $I_{\nu}(x)$ and $\mathbf{L}(x)$, respectively, we have that
\[\lim_{x\downarrow 0}\big(x\big(I_{\nu}(x)\mathbf{L}_{\nu-1}(x)-I_{\nu-1}(x)\mathbf{L}_{\nu}(x)\big)\big)=0, \qquad \mbox{for $\nu>-1/2$}.\]
Therefore, applying the indefinite integral formula (\ref{besint}) gives, for $\nu>-1/2$,
\begin{equation}\label{chicot}\int_0^xt^{\nu}I_{\nu}(t)\,\mathrm{d}t=\sqrt{\pi}2^{\nu-1}\Gamma(\nu+1/2)x(I_{\nu}(x)\mathbf{L}_{\nu-1}(x)-I_{\nu-1}(x)\mathbf{L}_{\nu}(x)).
\end{equation}
From inequalities (\ref{lowerint}) and (\ref{intIi}) of Theorem \ref{tiger}, we have
\[x^{\nu}I_{\nu+1}(x)<\int_0^xt^{\nu}I_{\nu}(t)\,\mathrm{d}t< \frac{2(\nu+1)}{2\nu+1}x^{\nu}I_{\nu+1}(x).\]
Substituting this inequality into (\ref{chicot}) gives
\begin{align*}x^{\nu}I_{\nu+1}(x)&<\sqrt{\pi}2^{\nu-1}\Gamma(\nu+1/2)x(I_{\nu}(x)\mathbf{L}_{\nu-1}(x)-I_{\nu-1}(x)\mathbf{L}_{\nu}(x)) \\
&<\frac{2(\nu+1)}{2\nu+1}x^{\nu}I_{\nu+1}(x).
\end{align*}
The desired inequality now follows from rearranging terms and an application of the standard formula $x\Gamma(x)=\Gamma(x+1)$.
\end{proof}

\begin{remark}The lower and upper bounds for $I_{\nu}(x)\mathbf{L}_{\nu-1}(x)-I_{\nu-1}(x)\mathbf{L}_{\nu}(x)$ that are given in Corollary \ref{struvebessel} are simple, but very tight for large $\nu$.
\end{remark}

\section{Inequalities for the modified Bessel function of the second kind}

We now present some simple inequalities for the modified Bessel function of the second kind $K_{\nu}(x)$.  The following theorem establishes an inequality for the modified Bessel function $K_{\nu}(x)$ that is useful in the study of Stein's method for Variance-Gamma distributions (see Gaunt \cite{me}).

\begin{theorem}\label{timet}
Let $\nu > 0$ and $x \geq 0$, then
\begin{equation} \label{frog} \frac{1}{x^2}-\frac{x^{\nu-2}K_{\nu}(x)}{2^{\nu-1}\Gamma(\nu)},
\end{equation}
is a monotone decreasing function of $x$ on $(0,\infty)$ and satisfies the following inequality
\begin{equation}\label{dog}0 < \frac{1}{x^2}-\frac{x^{\nu-2}K_{\nu}(x)}{2^{\nu-1}\Gamma(\nu)} \leq \frac{1}{4(\nu-1)}, \quad \mbox{for }x\geq 0,\: \nu>1.
\end{equation}
The lower bound is also valid for all $\nu>0$.
\end{theorem}

\begin{proof}
Applying the differentiation formula (\ref{cat}) gives
\begin{align}&\frac{\mathrm{d}}{\mathrm{d}x}\bigg(\frac{1}{x^2}-\frac{x^{\nu-2}K_{\nu}(x)}{2^{\nu-1}\Gamma(\nu)}\bigg) \nonumber \\
\label{100pgood}&=-\frac{2}{x^3} -\frac{(\nu-2)x^{\nu-3}K_{\nu}(x)-\frac{1}{2}(K_{\nu-1}(x)+K_{\nu+1}(x))x^{\nu-2}}{2^{\nu-1}\Gamma(\nu)}.
\end{align}
Using (\ref{Kidentity}) we may simplify the numerator as follows
\begin{align*}&(\nu-2)K_{\nu}(x)-\frac{1}{2} x(K_{\nu -1}(x)+K_{\nu +1}(x)) \\
 &= (\nu -2)K_{\nu}(x)-\frac{1}{2} x\left(2K_{\nu -1}(x)+\frac{2\nu}{x}K_{\nu}(x)\right) \\
&= -xK_{\nu -1}(x)-2K_{\nu}(x).
\end{align*}
Hence, (\ref{100pgood})  simplifies to
\[\frac{\mathrm{d}}{\mathrm{d}x}\bigg(\frac{1}{x^2}-\frac{x^{\nu-2}K_{\nu}(x)}{2^{\nu-1}\Gamma(\nu)}\bigg)=\frac{-2^{\nu}\Gamma(\nu)+x^{\nu+1}K_{\nu-1}(x)+2x^{\nu}K_{\nu}(x)}{2^{\nu-1}\Gamma(\nu)x^3}.\]
Thus, proving that (\ref{frog}) is monotone decreasing reduces to proving that, for $x>0$,
\begin{equation}\label{wolf} x^{\nu+1}K_{\nu-1}(x)+2x^{\nu}K_{\nu}(x) <2^{\nu}\Gamma(\nu).
\end{equation}
From (\ref{diffKi}) we get that
\begin{align*}\frac{\mathrm{d}}{\mathrm{d}x}\left(x^{\nu+1}K_{\nu-1}(x)+2x^{\nu}K_{\nu}(x)\right) &=\frac{\mathrm{d}}{\mathrm{d}x}\left(x^2 \cdot x^{\nu-1}K_{\nu-1}(x)+2x^{\nu}K_{\nu}(x)\right) \\
 &=2x^{\nu}K_{\nu-1}(x)-x^{\nu+1}K_{\nu-2}(x)-2x^{\nu}K_{\nu-1}(x) \\
 &=-x^{\nu+1}K_{\nu-2}(x) \\
 &<0.
\end{align*}
So $x^{\nu+1}K_{\nu-1}(x)+2x^{\nu}K_{\nu}(x)$ is a monotone decreasing function of $x$ and from the asymptotic formula (\ref{Ktend0}) we see that its limit as $x\to 0^{+}$ is $\lim_{x\to 0^{+}}(x^{\nu+1}K_{\nu-1}(x)+2x^{\nu}K_{\nu}(x)) =2\cdot 2^{\nu-1}\Gamma(\nu)=2^{\nu}\Gamma(\nu)$.  Therefore (\ref{wolf}) is proved, and so (\ref{frog}) is monotone decreasing on $(0,\infty)$.  It is therefore bounded above and below its values in the limits $x\rightarrow\infty$ and $x\rightarrow 0$.  These are calculated using the asymptotic formulas (\ref{Ktendinfinity}) and (\ref{K2ndterm}) and are given below:
\begin{eqnarray*}\lim_{x\to\infty}\bigg(\frac{1}{x^2}-\frac{x^{\nu-2}K_{\nu}(x)}{2^{\nu-1}\Gamma(\nu)}\bigg) &=&0, \\
\lim_{x\to 0^{+}}\bigg(\frac{1}{x^2}-\frac{x^{\nu-2}K_{\nu}(x)}{2^{\nu-1}\Gamma(\nu)}\bigg) &=&\frac{2^{\nu-3}\Gamma(\nu-1)}{2^{\nu-1}\Gamma(\nu)} =\frac{1}{4(\nu-1)},
\end{eqnarray*}
where the first limit holds for all $\nu>0$ and the second limit is valid for all $\nu>1$.  This completes the proof.
\end{proof}

\begin{remark}Inequality (\ref{dog}) of Theorem \ref{timet} is closely related to some inequalities given by Ismail \cite{ismail2} and Baricz et al$.$ \cite{baricz4}.  Ismail proved that $x^{\nu}K_{\nu}(x)\mathrm{e}^x>2^{\nu-1}\Gamma(\nu)$ for $x>0$, $\nu>1/2$, and Baricz et al$.$ proved that $x^{\nu-1}K_{\nu}(x)\geq 2^{\nu-1}\Gamma(\nu)K_1(x)$ for $x>0$, $\nu\geq 1$, which improves on the bound of Ismail for all $\nu\geq 1$.  From inequality (\ref{dog}) we can obtain lower and upper bounds for the quantity $x^{\nu}K_{\nu}(x)$.  The upper bound is $x^{\nu}K_{\nu}(x)<2^{\nu-1}\Gamma(\nu)$ for $x> 0$, $\nu>0$, and therefore, for $x>0$,
\[2^{\nu-1}\Gamma(\nu)\mathrm{e}^{-x}<x^{\nu}K_{\nu}(x)<2^{\nu-1}\Gamma(\nu), \quad \nu>0,\]
which can be improved as follows when $\nu\geq1$:
\[2^{\nu-1}\Gamma(\nu)\mathrm{e}^{-x}<2^{\nu-1}\Gamma(\nu)xK_1(x)\leq x^{\nu}K_{\nu}(x)<2^{\nu-1}\Gamma(\nu).\]
\end{remark}

Finally, we establish a simple, but surprisingly tight, lower bound for the modified Bessel function $K_0(x)$.

\begin{theorem}\label{beso}Let $x>0$, then
\begin{equation}\frac{\Gamma(x+1/2)}{\Gamma(x+1)}<\sqrt{\frac{2}{\pi}}\mathrm{e}^xK_0(x).
\end{equation}
\end{theorem}

\begin{proof}Formula 10.32.8 of Olver et al$.$ \cite{olver} gives the following integral representation of $K_0(x)$:
\[K_0(x)=\int_1^{\infty}\frac{\mathrm{e}^{-xt}}{\sqrt{t^2-1}}\,\mathrm{d}t, \qquad x>0.\]
Setting $t=2u+1$ gives
\[K_0(x)=\mathrm{e}^{-x}\int_0^{\infty}\frac{\mathrm{e}^{-2xu}}{\sqrt{u^2+u}}\,\mathrm{d}u.
\]
For $u>0$ we have $\mathrm{e}^{2u}-1=\sum_{k=1}^{\infty}\frac{(2u)^k}{k!}>2u+2u^2$, and so
\[\mathrm{e}^xK_0(x)>\sqrt{2}\int_0^{\infty}\frac{\mathrm{e}^{-2xu}}{\sqrt{\mathrm{e}^{2u}-1}}\,\mathrm{d}u=\sqrt{2}\int_0^{\infty}\frac{\mathrm{e}^{-(2x+1)u}}{\sqrt{1-\mathrm{e}^{-2u}}}\,\mathrm{d}u, \qquad \mbox{for } x>0.\]
Making the change of variables $y=\mathrm{e}^{-2u}$ gives
\begin{align*}\sqrt{2}\int_0^{\infty}\frac{\mathrm{e}^{-(2x+1)u}}{\sqrt{1-\mathrm{e}^{-2u}}}\,\mathrm{d}u&=\frac{1}{\sqrt{2}}\int_0^1(1-y)^{-1/2}y^{x-1/2}\,\mathrm{d}y \\
&=\frac{1}{\sqrt{2}}B(1/2,x+1/2) \\
&=\frac{\Gamma(1/2)\Gamma(x+1/2)}{\sqrt{2}\Gamma(x+1)} \\
&=\frac{\sqrt{\pi}\Gamma(x+1/2)}{\sqrt{2}\Gamma(x+1)},
\end{align*}
where $B(a,b)$ is the beta function, and we used the standard formula $B(a,b)=\frac{\Gamma(a)\Gamma(b)}{\Gamma(a+b)}$ to obtain the third equality.  This completes the proof.
\end{proof}

\begin{corollary}Let $x>0$, then
\[\frac{1}{\sqrt{x+1/2}}<\sqrt{\frac{2}{\pi}}\mathrm{e}^xK_0(x)<\frac{1}{\sqrt{x}}.\]
\end{corollary}

\begin{proof}The upper bound follows because $K_0(x)<K_{1/2}(x)=\sqrt{\frac{\pi}{2x}} \mathrm{e}^{-x}$.  The lower bound follows since $\frac{\Gamma(x+1/2)}{\Gamma(x+1)}>\frac{1}{\sqrt{x+1/2}}$, which we now prove.  Examining the proof of Theorem \ref{beso} we see that
\[\frac{\Gamma(x+1/2)}{\Gamma(x+1)}=\frac{2}{\sqrt{\pi}}\int_0^{\infty}\frac{\mathrm{e}^{-(2x+1)u}}{\sqrt{1-\mathrm{e}^{-2u}}}\,\mathrm{d}u.
\]
Now, for $u>0$ we have $1-\mathrm{e}^{-2u}=\sum_{k=1}^{\infty}(-1)^{k+1}\frac{(2u)^k}{k!}<2u$, and so
\[\frac{\Gamma(x+1/2)}{\Gamma(x+1)}>\frac{2}{\sqrt{\pi}}\int_0^{\infty}\frac{\mathrm{e}^{-(2x+1)u}}{\sqrt{2u}}\,\mathrm{d}u=\frac{2\sqrt{2}}{\sqrt{\pi}}\int_0^{\infty}\mathrm{e}^{-(2x+1)v^2}\,\mathrm{d}v=\frac{1}{\sqrt{x+1/2}},\]
as required
\end{proof}

\begin{remark}Luke \cite{luke} obtained the following bounds for $K_0(x)$:
\[\frac{8\sqrt{x}}{8x+1}<\sqrt{\frac{2}{\pi}}\mathrm{e}^xK_0(x)<\frac{16x+7}{(16x+9)\sqrt{x}}.\]
Numerical experiments show that the bounds of Luke and our lower bound of Corollary \ref{beso} are remarkably accurate for all but very small $x$, for which the logarithmic singularity of $K_0(x)$ blows up.  The lower bound $\frac{8\sqrt{x}}{8x+1}$ outperforms our bound lower bound of $\frac{\Gamma(x+1/2)}{\Gamma(x+1)}$ for $x>0.394$ (3 d.p.), whilst our bound outperforms for $x<0.394$ (3 d.p.), and performs considerably better for very small $x$.   
\end{remark}

\section*{Acknowledgements}During the course of this research the author was supported by an EPSRC DPhil Studentship, and an EPSRC Doctoral Prize. The author would like to thank Gesine Reinert for some fruitful discussions. Finally, the author would like to thank an anonymous referee for his/her helpful comments and suggestions.

\appendix
\section{Elementary of properties modified Bessel functions}

Here we list standard properties of modified Bessel functions that are used throughout this paper.  All these formulas can be found in Olver et al$.$ \cite{olver}, except for the inequalities and the integration formula (\ref{pdfk}), which can be found in Gradshetyn and Ryzhik \cite{gradshetyn}.

\subsection{Basic properties}
The modified Bessel functions $I_{\nu}(x)$ and $K_{\nu}(x)$ are both regular functions of $x$.  They satisfy the following simple inequalities
\begin{equation*}I_{\nu}(x) >0 \quad \mbox{for all }x>0, \: \mbox{for } \nu>-1,
\end{equation*}
\begin{equation*}K_{\nu}(x) >0 \quad \mbox{for all }x>0, \: \mbox{for all }\nu \in \mathbb{R}.
\end{equation*}

\subsection{Spherical Bessel functions}
\begin{equation} \label{sphere} K_{1/2}(x)=K_{-1/2}(x)=\sqrt{\frac{\pi}{2x}} \mathrm{e}^{-x}.
\end{equation}

\subsection{Asymptotic expansions}
\begin{align}\label{Itend0}I_{\nu} (x) &\sim \frac{1}{\Gamma(\nu +1)} \bigg(\frac{x}{2}\bigg)^{\nu}, \qquad x \downarrow 0, \: \nu>-1, \\
\label{Ktend0}K_{\nu} (x) &\sim \begin{cases} 2^{|\nu| -1} \Gamma (|\nu|) x^{-|\nu|}, & \quad x \downarrow 0, \: \nu \not= 0, \\
-\log x, & \quad x \downarrow 0, \: \nu = 0, \end{cases} \\
\label{K2ndterm}K_{\nu}(x)&\sim 2^{\nu-1}\Gamma(\nu)x^{-\nu}-2^{\nu-3}\Gamma(\nu-1)x^{-\nu+2}, \qquad x \downarrow 0, \: \nu >1, \\
\label{Ktendinfinity} K_{\nu} (x) &\sim \sqrt{\frac{\pi}{2x}} \mathrm{e}^{-x}, \quad x \rightarrow \infty, \\
\label{struve0}\mathbf{L}_{\nu}(x)&\sim \frac{2}{\sqrt{\pi}\Gamma(\nu+3/2)}\bigg(\frac{x}{2}\bigg)^{\nu+1}, \qquad x \downarrow 0, \: \nu>-1/2.
\end{align}

\subsection{Inequalities}
Let $x > 0$ then following inequalities hold
\begin{align}\label{Imon}I_{\nu} (x) < I_{\nu - 1} (x), \quad \nu \geq 1/2,\\
\label{Kmoni}K_{\nu} (x) < K_{\nu - 1} (x), \quad \nu < 1/2,\\
\label{cake}K_{\nu} (x) \geq K_{\nu - 1} (x), \quad \nu \geq 1/2.  
\end{align}
We have equality in (\ref{cake}) if and only if $\nu=1/2$.  The inequalities for $K_{\nu}(x)$ can be found in Ifantis and Siafarikas \cite{ifantis}, whilst the inequality for $I_{\nu}(x)$ can be found in Jones \cite{jones} and N\"{a}sell \cite{nasell}.  A survey of related inequalities for modified Bessel functions is given by Baricz \cite{baricz2}, and lower and upper bounds for the ratios $\frac{I_{\nu}(x)}{I_{\nu-1}(x)}$ and $\frac{K_{\nu}(x)}{K_{\nu-1}(x)}$ which improve on inequalities (\ref{Imon}) -- (\ref{cake}) are also given in Ifantis and Siafarikas \cite{ifantis} and Segura \cite{segura}.

\subsection{Identities}
\begin{align}\label{Iidentity}I_{\nu +1} (x) &= I_{\nu -1} (x) - \frac{2\nu}{x} I_{\nu} (x), \\
\label{Kidentity} K_{\nu +1} (x) &= K_{\nu -1} (x) + \frac{2\nu}{x} K_{\nu} (x).
\end{align}

\subsection{Differentiation}
\begin{align}\label{diffone}\frac{\mathrm{d}}{\mathrm{d}x} (x^{\nu} I_{\nu} (x) ) &= x^{\nu} I_{\nu -1} (x), \\
\label{diffKi}\frac{\mathrm{d}}{\mathrm{d}x} (x^{\nu} K_{\nu} (x) ) &= -x^{\nu} K_{\nu -1} (x), \\
\label{cat} \frac{\mathrm{d}}{\mathrm{d}x}(K_{\nu}(x))&=-\frac{1}{2}(K_{\nu+1}(x)+K_{\nu-1}(x)), \\
\label{cat111} \frac{\mathrm{d}}{\mathrm{d}x}(K_{\nu}(x))&=-K_{\nu-1}(x)-\frac{\nu}{x}K_{\nu}(x), \\
\label{cat222} \frac{\mathrm{d}}{\mathrm{d}x}(K_{\nu}(x))&=-K_{\nu+1}(x)+\frac{\nu}{x}K_{\nu}(x). 
\end{align}

\subsection{Integration}
\begin{equation} \label{pdfk} \int_{-\infty}^{\infty}\mathrm{e}^{\beta t} |t|^{\nu} K_{\nu}(|t|)\,\mathrm{d}t =\frac{\sqrt{\pi}\Gamma(\nu+1/2)2^{\nu}}{(1-\beta^2)^{\nu+1/2}}, \quad \nu>-1/2, \: -1<\beta <1.
\end{equation}

\bibliographystyle{amsplain}

\end{document}